\documentclass[11pt]{article}
\usepackage{epsf}
\usepackage{amsmath}
\usepackage{amsthm}
\usepackage[latin1]{inputenc}
\usepackage{graphicx,lscape}
\usepackage{amssymb}
\usepackage{amsfonts}
\usepackage[T1]{fontenc}
\usepackage{mathpazo}
\usepackage{float}
\usepackage{array}
\usepackage{epsfig}
\usepackage[english]{babel}
\usepackage{graphicx}

\allowdisplaybreaks[1] 

\newcounter{tab}[section]

\newcounter{fig}[section]

\newtheorem{defi}{Definition}[section]
\newtheorem{coro}[defi]{Corollary}

\newtheorem{theo}[defi]{Theorem} 

\newtheorem{remark}[defi]{Remark}
\newtheorem{example}[defi]{Example}

\numberwithin{equation}{section}

\newcounter{examp}[section]

\def\bfze{\mbox{\boldmath $0$}}
\def\bft{\mbox{\boldmath $t$}}
\def\bfx{\mbox{\boldmath $x$}}
\def\bfa{\mbox{\boldmath $a$}}
\def\bfb{\mbox{\boldmath $b$}}
\def\bfz{\mbox{\boldmath $z$}}
\def\bfm{\mbox{\boldmath $m$}}
\def\bfs{\mbox{\boldmath $s$}}
\def\bfI{\mbox{\boldmath $I$}}

\begin{document}
\thispagestyle{empty}
\noindent
{\bf \Large 
A note on gaussian distributions in $\mathbb{R}^n$
}
\vspace{0.5cm} \\
{
B.G. MANJUNATH and K.R. PARTHASARATHY 
}
\vspace{0.5cm} \\
{\small 
Indian Statistical Institute, Delhi Centre, 7, S. J. S. Sansanwal Marg, \\
New Delhi -- 110016, India \\
E--mail : bgmanjunath@gmail.com; krp@isid.ac.in 
}
\vspace{0.4cm}

\noindent
{\small
{\bf Abstract.} Given any finite set $\mathcal{F}$ of $(n-1)$-- dimensional subspaces of $\mathbb{R}^n$ we give
examples of nongaussian probability measures in $\mathbb{R}^n$ whose marginal distribution in each subspace from
$\mathcal{F}$ is gaussian.  However, if $\mathcal{F}$ is an infinite family of such $(n-1)$--dimensional subspaces then
such a nongaussian probability measure in $\mathbb{R}^n$ does not exist.
}

\vspace{0.4cm}


\vspace{0.4cm}
\noindent
{\small
{AMS 1991} subject classification: primary, 60G15, 60E10; secondary, 62E15
}


\section{Introduction}
Starting with the simple example of E. Nelson as cited by W. Feller in \cite{wf66} we have from the papers of B.K. Kale \cite{k70}, G.G. Hamedani and
M.N. Tata \cite{ht75} and Y. Shao and M. Zhou \cite{sz10} etc., as well as Section 10 of
J. Stoyanov's book \cite{s87}, several examples of bivariate and multivariate nongaussian distributions
under which many linear functionals can have a gaussian distribution on the real line.  These results suggest the possibility
of characterizing a gaussian distribution in $\mathbb{R}^n$ through properties of classes of linear functionals.  Motivated by Nelson's
example in \cite{wf66} and the bivariate construction in \cite{ht75} we introduce a perturbation of the standard
gaussian density function in $\mathbb{R}^n$ exhibiting the following interesting features: (1) Given any finite set
$\{S_j, 1 \leq j \leq N \}$ of $(n-1)$--dimensional subspaces it has a marginal density function which is standard gaussian in each 
$S_j$, $j\in\{1,2,...,N\}$; (2) There can exist linear functionals whose distributions may have nonunimodal density functions; (3) For certain
choices of subspaces the nongaussian perturbation can be so chosen that any real symmetric measurable function of all the 
$n$ coordinates has its distribution preserved.  In particular, the sum of squares of all the coordinates can have the 
$\chi^2$ distribution with $n$ degrees of freedom.

We also demonstrate the following characterization of the multivariate gaussian distribution.  Suppose $\{S_j, j=1,2,...\}$ is a
countably infinite set of $(n-1)$--dimensional subspaces of $\mathbb{R}^n$ and $\mu$ is a probability measure in $\mathbb{R}^n$
such that the projection of $\mu$ in each subspace $S_j$ is gaussian.  Then $\mu$ itself is gaussian.  This is a generalization
of the characterization in \cite{ht75} and a more precise version of the result in \cite{sz10}.

Our proofs follow the steps in \cite{ht75} and use some additional geometric and topological arguments of a very elementary kind.

\section{A perturbation of the gaussian \\ characteristic function}\label{intro}
We begin by examining a small perturbation of the characteristic function of the $n$--variate standard gaussian
distribution with mean vector $\bfze$ and covariance matrix $\bfI$ as follows.  Choose and fix any homogeneous polynomial $\mathcal{P}$
of even degree $2k$ in $n$ real variables $t_1,t_2,...,t_n$ and define

\begin{eqnarray} \label{1.1}
\Phi (\bft;\varepsilon,\sigma,\mathcal{P})   = e^{-\frac{1}{2} |\bft|^2} + \varepsilon \, \, e^{-\frac{1}{2} \sigma^2 |\bft|^2} \mathcal{P}(\bft), \mbox{  } 
\bft \in \mathbb{R}^n
\end{eqnarray}

\noindent
where $\bft = (t_1,...,t_n)^T$, $\varepsilon$ is a real parameter and $\sigma$ is a parameter satisfying $0<\sigma<1$.
Here
\begin{eqnarray*}
 |\bft|^2 = (t^2_1 + ... + t^2_n).
\end{eqnarray*}

\noindent
Clearly, $\Phi(\cdot;\varepsilon,\sigma,\mathcal{P})$ is a real analytic function on $\mathbb{R}^n$ satisfying

\begin{eqnarray} \label{1.2}
\Phi(\bfze;\varepsilon , \sigma,\mathcal{P})  &=& 1,   \nonumber \\ 
\Phi(-\bft;\varepsilon , \sigma,\mathcal{P})  &=& \Phi(\bft; \varepsilon , \sigma,\mathcal{P}). 
\end{eqnarray}

\noindent
Let 
\begin{eqnarray} \label{1.3}
 Z_\mathcal{P} = \{ \bft | \mathcal{P}(\bft) = 0, \bft \in \mathbb{R}^n \}
\end{eqnarray}
be the set of zeros of $\mathcal{P}$ in $\mathbb{R}^n$.

Since we are interested in the inverse Fourier transform of $\Phi$ we introduce the renormalized polynomial
$:\mathcal{P}:$  in the form of a formal definition.

\begin{defi}
 Let 
\begin{eqnarray*}
\mathfrak{N}(x) = \frac{1}{\sqrt{2\pi}} e^{-\frac{1}{2}x^2}
\end{eqnarray*}
and let $H_m(x)$ be the $m$-th Hermite polynomial defined by
\begin{eqnarray*}
 \frac{d^m}{dx^m} \mathfrak{N}(x) = (-1)^m H_m(x) \mathfrak{N}(x), \mbox{   } m = 0,1,2,...
\end{eqnarray*}
(as in Feller \cite{wf66}).  For any real polynomial $\mathcal{P}$ in $n$ real variables given by
\begin{eqnarray*}
 \mathcal{P}(t_1,t_2,...,t_n) = \sum_{\bfm}  \, \, a_{m_1,m_2,...,m_n} t^{m_1}_1 t^{m_2}_2 ... t^{m_n}_n
\end{eqnarray*}
its renormalized version $:\mathcal{P}:$ is defined by
\begin{eqnarray*}
 :\mathcal{P}:(x_1,...,x_n) = \sum_{\bfm} \, \,a_{m_1,m_2,...,m_n} H_{m_1} (x_1) H_{m_2} (x_2)...H_{m_n} (x_n).
\end{eqnarray*}
Note that for a homogeneous polynomial, its renormalized version need not be homogeneous.
\end{defi}

Since the function $\Phi$ in (\ref{1.1}) is in $\mathbb{L}_1 (\mathbb{R}^n)$ its inverse Fourier transform 
$f$ is defined by
\begin{eqnarray}\label{1.4}
f(\bfx; \varepsilon,\sigma,\mathcal{P})  &=& \frac{1}{(2\pi)^n} \int e^{-i\bft^T \bfx} \Phi(\bft ; \varepsilon,\sigma,\mathcal{P}) 
dt_1 dt_2 ... dt_n  \nonumber \\
&=& \frac{1}{(\sqrt{2\pi})^n} e^{-\frac{1}{2}|\bfx|^2} + \varepsilon \, \,  \frac{1}{(2\pi)^n} \int e^{-i \bft^T \bfx} 
e^{- \frac{1}{2} \sigma^2 |\bft|^2}  \mathcal{P}(\bft) dt_1 ... dt_n. \nonumber \\
\end{eqnarray}
First, we note that

\begin{eqnarray*}
 \frac{1}{(2\pi)^n} \int e^{-i \bft^T \bfx} 
e^{- \frac{1}{2} \sigma^2 |\bft|^2}  dt_1 dt_2 ... dt_n =  \frac{1}{\sigma^n} \prod^n_{j=1} \mathfrak{N} \left( \frac{x_j}{\sigma}  \right).
\end{eqnarray*}
Repeated differentiation with respect to $x_1,x_2,...,x_n$ shows that for the homogeneous polynomial $\mathcal{P}$
of degree $2k$ we have
 
\begin{eqnarray*}
 && \frac{1}{(2\pi)^n} \int e^{-i \bft^T \bfx} 
e^{- \frac{1}{2} \sigma^2 |\bft|^2}  \mathcal{P}(\bft) dt_1 dt_2 ... dt_n  \\
 && = \frac{1}{\sigma^n} \mathcal{P} \left( i\frac{\partial}{\partial x_1}
,...,i\frac{\partial}{\partial x_n} \right ) \left\{ \prod^n_{j=1}\mathfrak{N} \left( \frac{x_j}{\sigma}  \right) \right\} \\
&& =  \frac{(-1)^k}{\sigma^{n+2k}} :\mathcal{P}:\left( \frac{x_1}{\sigma},...,\frac{x_n}{\sigma} \right) 
 \frac{1}{(\sqrt{2\pi})^n} e^{-\frac{1}{2 \sigma^2} |\bfx|^2}. 
\end{eqnarray*}

Thus the inverse Fourier transform (\ref{1.4}) assumes the form 
\begin{eqnarray} \label{1.5}
&& f(\bfx; \varepsilon,\sigma,\mathcal{P}) \nonumber \\
&&=\frac{1}{(\sqrt{2\pi})^n} e^{-\frac{1}{2} |\bfx|^2} \left\{ 1 +  \frac{(-1)^k \, \varepsilon}{\sigma^{n+2k}}
:\mathcal{P}:\left( \frac{x_1}{\sigma},...,\frac{x_n}{\sigma} \right)  e^{-\frac{1}{2\sigma^2} |\bfx|^2 (1-\sigma^2)}  \right\}. \nonumber \\
\end{eqnarray}

Since, by assumption, $1-\sigma^2 > 0$ the positive constant $K(\sigma,\mathcal{P})$ defined by
\begin{eqnarray} \label{1.6}
 K(\sigma,\mathcal{P}) = \sup_{\bfx \in \mathbb{R}^n} \frac{|:\mathcal{P}:(x_1,...,x_n)|}{\sigma^{n+2k}} e^{-\frac{1}{2}|\bfx|^2 (1-\sigma^2) }
\end{eqnarray}
is finite and for all $\bfx \in \mathbb{R}^n$
\begin{eqnarray*}
 f(\bfx; \varepsilon,\sigma,\mathcal{P}) \geq 0  \mbox{  if  }  |\varepsilon|  \leq K^{-1}(\sigma,\mathcal{P})
\end{eqnarray*}
we observe that $\Phi(\cdot; \varepsilon,\sigma,\mathcal{P})$ is a real characteristic function of the probability 
density function $f(\cdot; \varepsilon,\sigma,\mathcal{P})$ defined by (\ref{1.5}) for any \\
$\varepsilon \in [-K^{-1}(\sigma,\mathcal{P}),K^{-1}(\sigma,\mathcal{P})]$.
Here we have made use of property (\ref{1.2}).  Thus we can summarize the discussion above as a theorem. 

\begin{theo}
Let $0<\sigma<1$, $\mathcal{P}$ be a real homogeneous polynomial in $n$ variables of even degree $2k$, $K(\sigma,\mathcal{P})$ the positive constant defined
 by (\ref{1.6}) and $\varepsilon \in [-K^{-1}(\sigma,\mathcal{P}),K^{-1}(\sigma,\mathcal{P})]$.  Then the function 
$\Phi(\cdot; \varepsilon,\sigma,\mathcal{P} )$ defined by (\ref{1.1}) is the characteristic function of a probability density function
$f(\cdot;\varepsilon,\sigma,\mathcal{P})$ defined by (\ref{1.5}).  Under this density function $f(\cdot;\varepsilon,\sigma,\mathcal{P})$
the linear functional  $\bfx \longmapsto \bfa^T \bfx$ with $|\bfa|=1$ has characteristic function $\varphi_{\bfa}$ and probability density function $g_{\bfa}$ on the real line given respectively by
\begin{eqnarray} \label{1.7}
 \varphi_{\bfa}(t) = e^{-\frac{1}{2} t^2} + \varepsilon \, \, \mathcal{P}(\bfa) e^{-\frac{1}{2} \sigma^2 t^2}  t^{2k}, 
\mbox{   } t \in \mathbb{R}
\end{eqnarray}

\begin{eqnarray} \label{1.8}
 g_{\bfa}(x) =  \frac{1}{\sqrt{2\pi}} \left \{ e^{- \frac{1}{2} x^2} + \frac{(-1)^k \, \varepsilon \, \mathcal{P}(\bfa) }{\sigma^{2k+1}} 
H_{2k} \left(\frac{x}{\sigma}\right) e^{-\frac{1}{2\sigma^2} x^2 }  \right \}.  
\end{eqnarray}

In particular, for any $\bfa \in Z_{\mathcal{P}}$, the linear functional $\bfa^T \bfx$ has the normal distribution
with mean $0$ and variance  $|\bfa|^2$ but $f(\cdot; \varepsilon,\sigma, \mathcal{P})$ is a nongaussian 
density function for any $\varepsilon \in [-K^{-1}(\sigma,\mathcal{P}),K^{-1}(\sigma,\mathcal{P})  ] \smallsetminus \{0\}$.
\end{theo}

\begin{proof}
The first part is immediate from the discussion preceding the statement of the theorem.  To prove the second
part we note that the characteristic function $\varphi_{\bfa}(t)$ of the linear functional $\bfa^T\bfx$ under the density
function $f(\cdot; \varepsilon,\sigma,\mathcal{P})$ is $\Phi(t\bfa; \varepsilon,\sigma,\mathcal{P})$ and (\ref{1.7}) follows from 
(\ref{1.1}) and the homogeneity of $\mathcal{P}$.  Now (\ref{1.8}) follows from Fourier inversion of (\ref{1.7}).
If $0 \neq \bfa \in Z_{\mathcal{P}}$ then  $0 = \mathcal{P}(\bfa) = \mathcal{P}\left(\frac{\bfa}{|\bfa|}\right)$ and therefore 

\begin{eqnarray*}
 \varphi_{\frac{\bfa}{|\bfa|} }(t) =  e^{-\frac{1}{2} t^2}.
\end{eqnarray*}

\noindent
Hence $\bfa^T \bfx$ is normally distributed with mean $0$  and variance $|\bfa|^2$.
\end{proof}

\begin{coro}
 Let $\{S_j,  1\leq j\leq N\}$ be any finite set of $(n-1)$--dimensional subspaces of $\mathbb{R}^n$.  Then there
exists a nongaussian analytic probability density function whose projection on $S_j$ is gaussian for each 
$j \in \{1,2,...,N\}$.
\end{coro}

\begin{proof}
 By adding one more $(n-1)$--dimensional subspace to the collection $\{S_j,  1\leq j\leq N \}$, if necessary, we may 
assume without loss of generality that $N$ is even.  Choose a unit vector $\bfa^{(j)} \in S^\bot_j$ for each $j$ and define
the homogeneous real polynomial $\mathcal{P}$ of degree $N$ by
\begin{eqnarray*}
 \mathcal{P}(\bft) = \prod^{N}_{j=1}  \bfa^{(j)^T} \bft, \mbox{   } \bft \in \mathbb{R}^n. 
\end{eqnarray*}
Clearly,
\begin{eqnarray*}
 \mathcal{P}(\bft) = 0  \mbox{      if    }  \bft \in \bigcup^{N}_{j=1}  S_j.
\end{eqnarray*}
In other words
\begin{eqnarray*}
 \bigcup^{N}_{j=1} S_j \subset Z_{\mathcal{P}}.
\end{eqnarray*}
If we choose $\mu$ to be the probability measure with the density function $f(\cdot ; \varepsilon,\sigma,\mathcal{P})$,
$0 \neq \varepsilon$ in $[-K^{-1}(\sigma,\mathcal{P}), K^{-1}(\sigma,\mathcal{P})]$ in Theorem 2.2 it 
follows immediately from the last part
of the theorem that every linear functional of the form $\bfb^T \bfx$ has a normal distribution with mean $0$ and
variance $|\bfb|^2$ whenever $\bfb \in Z_{\mathcal{P}}$.  This completes the proof.
\end{proof}

\begin{remark}
 In the context of understanding the modes of the density function $g_{\bfa}(x)$ given by (\ref{1.8}) it is of interest to
note that 
\begin{eqnarray*}
 && \left\{ x  \Bigl |  x  \Bigr. \neq 0, g'_{\bfa} (x) = 0 \right \} = \\
&& \left \{ x  \Bigl |  x  \Bigr. \neq 0, e^{\frac{x^2}{2}(\frac{1}{\sigma^2} -1)}  +  \frac{ (-1)^k \, \varepsilon \mathcal{P} (\bfa) }{\sigma^{2k+2}}  
\frac{H_{2k+1}(\frac{x}{\sigma})}{x} = 0 \right \}.
\end{eqnarray*}
Indeed, this is obtained by straightforward differentiation and using the recurrence relation 
$H_{2k+1}(x) = x H_{2k}(x) - H'_{2k}(x)$.
\end{remark}

\begin{example}
\normalfont
{
 Let $n$ be even,

\begin{eqnarray} \label{1.9}
 \mathcal{P}(t_1,t_2,...,t_n)  &=& t_1 t_2...t_n \prod_{i >j} (t^2_i - t^2_j) \nonumber \\
&=& t_1 t_2...t_n 
\begin{vmatrix}
  1     &   1      &  ...  & 1     \\
  t^2_1 &   t^2_2  &  ...  & t^2_n \\
  t^4_1 &   t^4_2  &  ...  & t^4_n \\
   .    &   .      &  ...  & .      \\
   .    &   .      &  ...  & .      \\
   .    &   .      &  ...  & .      \\
  t^{2(n-1)}_1 &   t^{2(n-1)}_2  &  ...  & t^{2(n-1)}_n 
\end{vmatrix}
\end{eqnarray}
Then $\mathcal{P}$ is a polynomial of even degree $n^2$, which is antisymmetric in the variables
$t_1, t_2,...,t_n$.  The renormalized version $:\mathcal{P}:$ of $\mathcal{P}$ is given by
\begin{eqnarray} \label{1.10}
:\mathcal{P}: (x_1,x_2,...,x_n) = 
\begin{vmatrix}
  H_1(x_1)     &   H_1(x_2)      &  ...  & H_1(x_n)     \\
  H_3(x_1)     &   H_3(x_2)      &  ...  & H_3(x_n) \\
   .    &   .      &  ...  & .      \\
   .    &   .      &  ...  & .      \\
   .    &   .      &  ...  & .      \\
  H_{2n+1}(x_1)     &   H_{2n+1}(x_2)      &  ...  & H_{2n+1}(x_n) 
\end{vmatrix}
\end{eqnarray}

In particular, $:\mathcal{P}:$ is antisymmetric in the variables $x_1,x_2,...,x_n$.  Fixing $0<\sigma<1$ we get for each
$\varepsilon \in [ -K^{-1}(\sigma,\mathcal{P}), K^{-1}(\sigma,\mathcal{P})]$, with $K(\sigma,\mathcal{P})$ being determined
 by (\ref{1.6}),
(\ref{1.10}) and $k=\frac{1}{2}n^2$, the probability density function $f(\cdot ; \varepsilon,\sigma,\mathcal{P})$ given by

\begin{eqnarray} \label{1.11}
 f(\bfx; \varepsilon,\sigma,\mathcal{P})  &=& \frac{1}{(\sqrt{2\pi})^n} \left\{  e^{-\frac{1}{2} |\bfx|^2} \right. \nonumber \\
&& +  \frac{\varepsilon}{\sigma^{n(n+1)}}
\begin{vmatrix}
  H_1(\frac{x_1}{\sigma})     &   H_1(\frac{x_2}{\sigma})      &  ...  & H_1(\frac{x_n}{\sigma})     \\
  H_3(\frac{x_1}{\sigma})     &   H_3(\frac{x_2}{\sigma})      &  ...  & H_3(\frac{x_n}{\sigma}) \\
   .    &   .      &  ...  & .      \\
   .    &   .      &  ...  & .      \\
   .    &   .      &  ...  & .      \\
  H_{2n+1}(\frac{x_1}{\sigma})     &   H_{2n+1}(\frac{x_2}{\sigma})      &  ...  & H_{2n+1}(\frac{x_n}{\sigma})
\end{vmatrix} \nonumber \\
   && \mbox{    } \left. e^{-\frac{1}{2} \sigma^2 |\bfx|^2} \right \}.  
\end{eqnarray}

By Theorem 2.2 and its Corollary we conclude that the projection of this density function on the $(n-1)$--dimensional
hyperplanes $\{\bfx | x_j = 0 \}$, $1 \leq j \leq n$; $\{\bfx | x_i - x_j = 0 \}$, $1 \leq i \leq j \leq n$ and
$\{\bfx | x_i + x_j =  0 \}$, $1 \leq i \leq j \leq n$ are all $(n-1)$--dimensional gaussian densities.

If $g(x_1,x_2,...,x_n)$ is any bounded continuous function which is symmetric in the variables $x_1,x_2,...,x_n$ then
the function $g:\mathcal{P}:$ is an antisymmetric function in $\mathbb{R}^n$ and therefore

\begin{eqnarray*} 
 \int _{\mathbb{R}^n}  (g:\mathcal{P}:) (x_1,x_2,...,x_n) e^{-\frac{1}{2\sigma^2} |\bfx|^2} dx_1 dx_2 ... dx_n = 0.
\end{eqnarray*}
Thus
\begin{eqnarray*} 
 &&\int _{\mathbb{R}^n}  g(x_1,...,x_n) f(\bfx; \varepsilon,\sigma,\mathcal{P}) dx_1 dx_2 ... dx_n  \\
&& = \int _{\mathbb{R}^n}  g(x_1,x_2,...,x_n) \frac{1}{(\sqrt{2\pi})^n} e^{-\frac{1}{2} |\bfx|^2} dx_1 dx_2 ... dx_n. 
\end{eqnarray*}
In other words, for $0 \neq \varepsilon \in [-K^{-1}(\sigma,\mathcal{P}), K^{-1}(\sigma,\mathcal{P})]$, any symmetric measurable
function $g$ on $\mathbb{R}^n$ has the property that its distribution under the nongaussian density function 
$f(\bfx;\varepsilon,\sigma, \mathcal{P} )$ in (\ref{1.11}) is the same as its distribution under the standard gaussian density function
with mean $\bfze$  and covariance matrix $\bfI$.
}
\end{example}

\begin{example}
\normalfont
{
 We now specialize Example 2.5 to the case $n=2$, $\sigma = 2^{-1/2}$ when
\begin{eqnarray*} 
 \mathcal{P} (t_1,t_2) &=& t_1 t_2 ( t^2_1 - t^2_2), \\
:\mathcal{P}:(x_1,x_2) &=&  H_3(x_1) H_1(x_2) - H_1(x_1) H_3(x_2) \\
 &=& x^3_1 x_2 - x^3_2 x_1.
\end{eqnarray*}
A simple computation shows that

\begin{eqnarray*} 
K(\sigma,\mathcal{P}) &=&  8 \sup |x^3_1 x_2 - x^3_2 x_1| \, \, e^{-\frac{1}{4} (x^2_1 + x^2_2)} \\
&=& 128 \, \, e^{-2}.
\end{eqnarray*}
This supremum is easily evaluated by switching over to the polar coordinates $x_1 = r cos \theta$, $x_2 = r sin \theta$.
Then
 \begin{eqnarray} \label{1.12}
  f(\bfx; \varepsilon,\sigma,\mathcal{P}) = \frac{1}{2\pi} e^{-\frac{1}{2}(x^2_1 + x^2_2)} \left\{ 1 +
  32 \varepsilon (x^3_1  x_2 - x^3_2 x_1) e^{-\frac{1}{2}(x^2_1 + x^2_2)}  \right \}
 \end{eqnarray}
which is a probability density function whenever
 \begin{eqnarray*} 
  |\varepsilon| \leq \frac{e^2}{128}.
 \end{eqnarray*}
At $\varepsilon = 0$, it is the standard normal density function with mean $\bfze$ and covariance matrix $\bfI$.  
We write $\eta = 32 \, \, \varepsilon$
and express the density function (\ref{1.12}) as
 \begin{eqnarray} \label{1.13}
f_{\eta}(x_1,x_2) = \frac{1}{2\pi} e^{-\frac{1}{2}(x^2_1 + x^2_2)} \left\{ 1 +
  \eta (x^3_1  x_2 - x^3_2 x_1) e^{-\frac{1}{2}(x^2_1 + x^2_2)}  \right \}  
 \end{eqnarray}
where (See Fig. (\ref{bivariate}).)
 \begin{eqnarray*}
  |\eta|  \leq \frac{e^2}{4}.
 \end{eqnarray*}

\begin{figure}[h!]  
  \centering
    \includegraphics[width=0.8\textwidth]{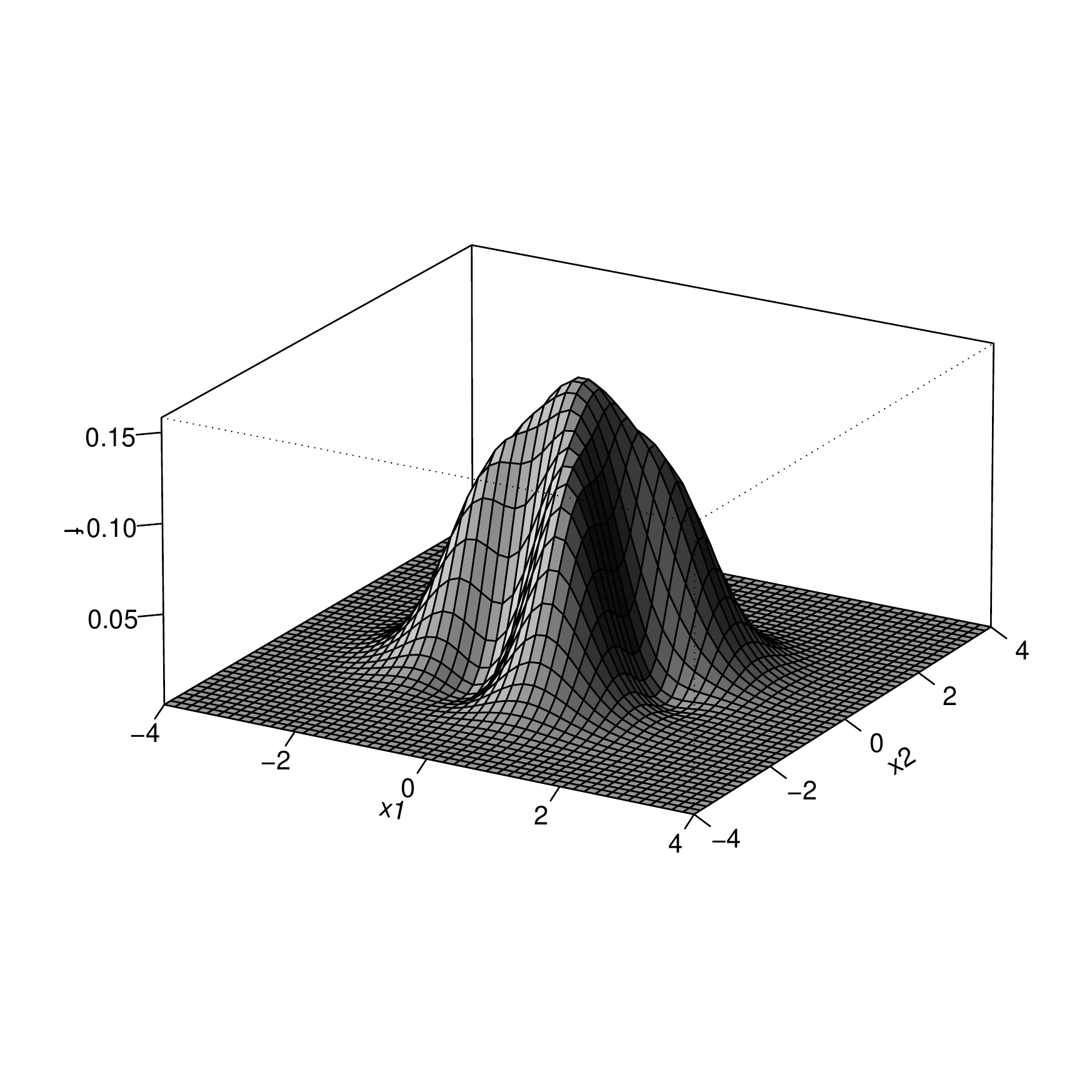}
    \caption{Bivariate density  $f_{\eta}(x_1,x_2)$ at $\eta = e^2 / 4$.}
    \label{bivariate}
\end{figure}

\noindent
When $\bfa = ( sin \theta, cos \theta)$ the density function $g_{\theta}$ of the linear functional
$\bfx \longmapsto x_1 sin  \theta + x_2  cos \theta$, under $f_{\eta}$ is given by the formula (\ref{1.8})
of Theorem 2.2 as
\begin{eqnarray}  \label{1.14}
 g_{\theta} (x) = \frac{1}{\sqrt{2 \pi}} e^{-\frac{1}{2} x^2}  \left\{ 1 - \frac{\sqrt{2} \,\, \eta \,\, sin (4\theta)}{32} (4 x^4 -12 x^2 +3) 
e^{-\frac{1}{2} x^2}  \right\}.
\end{eqnarray}

\noindent
Thus
\begin{eqnarray*}  
 g'_{\theta} (x) = \frac{-x}{\sqrt{2 \pi}} e^{-\frac{1}{2} x^2}  \left\{ e^{\frac{1}{2} x^2} - 
\frac{\sqrt{2} \,\, \eta  \, \,  sin (4\theta) }{16} (4 x^4 -20x^2 + 15) e^{-\frac{1}{2}x^2}  \right\}.
\end{eqnarray*}
It is not difficult to find values of $\eta$ in the range $(0, \frac{1}{4}e^2]$ and angle $\theta$ for which

\begin{eqnarray}  
 \left\{ x \Bigl |    \Bigr. e^{\frac{1}{2}x^2}  - \frac{\sqrt{2} \,\, \eta \,\, sin(4\theta)}{16}   ( 4x^4 -20x^2 +15) = 0,  x \neq 0 \right\} \neq \emptyset.
\end{eqnarray}
This reveals the possibility of nonunimodality of the density of some linear functionals under the joint density
$f_{\eta}$. For an illutration cf. Fig. (\ref{bimodel}).

\begin{figure}[h!]  
  \centering
    \includegraphics[width=0.8\textwidth]{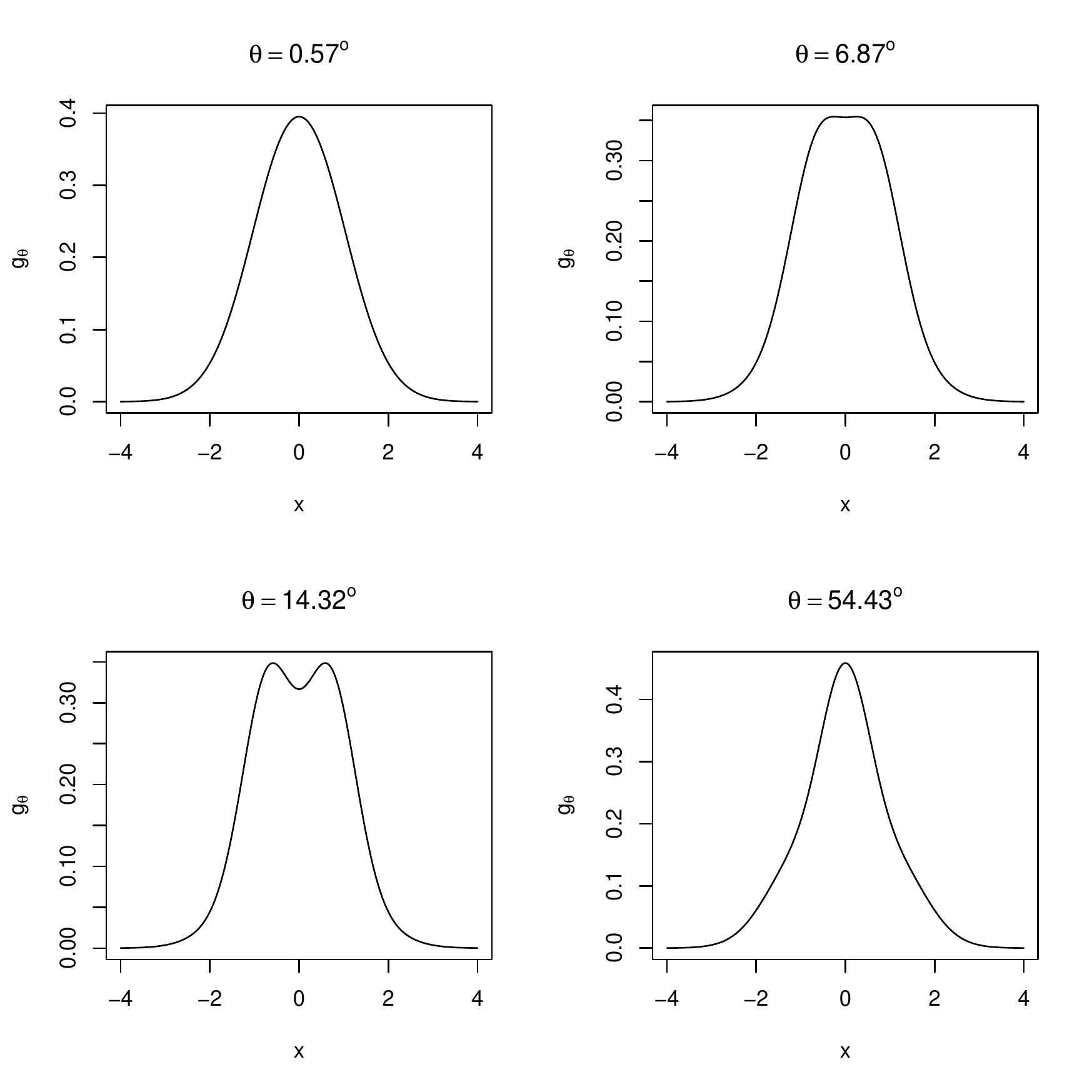}
    \caption{Nonunimodality of $g_{\theta}$. }
    \label{bimodel}
\end{figure}
}
\end{example}

\section{A characterization of gaussian distributions in $\mathbb{R}^n$}
In the context of the Corollary to Theorem 2.2 we have the following characterization of a 
gaussian distribution in $\mathbb{R}^n$ when the number $N$ of $(n-1)$--dimensional subspaces in
the corollary is countably infinite.

\begin{theo}
 Let $\{S_j, j=1,2...\}$ be a countably infinite set of $(n-1)$-- dimensional subspaces of $\mathbb{R}^n$ and let $\mu$
be a probability measure in $\mathbb{R}^n$ whose projection on $S_j$ is gaussian for each $j=1,2,...$ . Then $\mu$ is gaussian.

\begin{proof}
 The fact that the projection of $\mu$ on the two distinct $(n-1)$--dimensional subspaces $S_1$ and $S_2$ are
gaussian implies that the multivariate Laplace transform $\widehat{\mu}$ of $\mu$ given by
\begin{eqnarray} \label{muhat}
 \widehat{\mu}(z_1,...,z_n) =  \int \exp (z_1 x_1 + ... + z_n x_n) \mu(dx_1 dx_2 ... dx_n)
\end{eqnarray}
is well-defined for $\bfz \in \mathbb{C}^n$ and analytic in each of the complex variables $z_j$, $j=1,...,n$. Let $\bfm$ and $\Sigma$
be respectively the mean vector and covariance matrix of the $\mathbb{R}^n$ valued random variable $\bfx$ with distribution $\mu$.

Choose and fix a unit vector $\bfa^{(j)} \in S^\bot_j$ for each $j=1,2,...$ .   Suppose
\begin{eqnarray*}
 \bfa^{(j)^T} &=& (a_{j1},...,a_{jn}), \mbox{   }  j=1,2,...,  \\
\alpha_j &=&  \max_{1 \leq r \leq n } |a_{jr}|.
\end{eqnarray*}
Since 
\begin{eqnarray*}
 \sum^n_{r=1} a^2_{jr} = 1,   \mbox{        }    \forall  j
\end{eqnarray*}
it follows that $\alpha_j \geq n^{-1/2}$,   $\forall$ $j$.  There exists an $r_0$ such that $a_{jr_0} = \alpha_j$
for infinitely many values of $j$. Restricting ourselves to this infinite set of $j$`s and assuming $r_0 =1$ without
loss of generality we may as well assume that
\begin{eqnarray*}
 \bfa^{(j)} &=& (a_{j1},..., a_{jn})^T,  \\
  |a_{j1}|&=&  \max_{1 \leq r \leq n} |a_{jr}|  \mbox{       } \forall \mbox{    } j=1,2,..., \\
 |a_{j1}| &\geq& n^{-1/2}  \mbox{      }  \forall j.
\end{eqnarray*}
Now consider the $(n-1)$-- dimensional vector $\bfb^{(j)}$ defined by

\begin{eqnarray*}
 \bfb^{(j)^T} = \left(\frac{a_{j2}}{a_{j1}},\frac{a_{j3}}{a_{j1}},...,\frac{a_{jn}}{a_{j1}} \right), \mbox{       } j = 1,2,....
\end{eqnarray*}
where
\begin{eqnarray*}
 \left| \frac{a_{jr}}{a_{j1}} \right | \leq 1  \mbox{     }  \forall  \mbox{     } r=2,3,...,n.
\end{eqnarray*}
Thus all the vectors $\bfb^{(j)}$ are distinct and they constitute a bounded countable set in $\mathbb{R}^{(n-1)}$.
Define the set
\begin{eqnarray*}
 \mathbb{D} = \bigcap_{j<i} \left\{ \bfs \Bigl | \Bigr. \bfs \in \mathbb{R}^{(n-1)}, (\bfb^{(j)} - \bfb^{(i)})^T \bfs \neq 0 \right\}.
\end{eqnarray*}
Being a countable intersection of dense open sets it follows from the Baire category theorem that $\mathbb{D}$ is dense
in $\mathbb{R}^{(n-1)}$.  Let now
\begin{eqnarray*}
 \bfs = (s_2,s_3,...,s_n)^T  \in \mathbb{R}^{(n-1)}
\end{eqnarray*}
be any fixed point in $\mathbb{D}$.  Define
\begin{eqnarray*}
 s_{j1} = - \bfb^{(j)^T} \bfs,  \mbox{     }  j=1,2,... \mbox{   } .
\end{eqnarray*}
By the definition of $\mathbb{D}$, $\{ s_{j1}, j =1,2,...\}$ is a bounded and countably infinite set of distinct points on the real line.
Furthermore
\begin{eqnarray*}
 a_{j1}s_{j1} +a_{j2}s_{2} + ... + a_{jn} s_{n} = 0 \mbox{      } \forall \mbox{      } j.
\end{eqnarray*}
In other words, $(s_{j1},s_2,...,s_n)^T \in S_j$ for each $j$.  By hypothesis the linear functional $s_{j1}x_1 + s_2 x_2 +...+ s_n x_n$
has a normal distribution with mean $s_{j1} m_1 + s_2 m_2 +...+s_n m_n$ and variance $(s_{j1},s_2,...,s_n) \Sigma (s_{j1},s_2,...,s_n)^T$.
Defining
\begin{eqnarray*}
 \psi (z_1,...,z_n) = \exp ( \bfm^T \bfz + \frac{1}{2} \bfz^T \Sigma \bfz ), \mbox{      } \bfz \in \mathbb{C}^n
\end{eqnarray*}
we conclude that the Laplace transform $\widehat{\mu}$ defined by (\ref{muhat}) and the function $\psi$ satisfy the relation
\begin{eqnarray*}
 \widehat{\mu} (s_{j1},s_2,...,s_n) = \psi (s_{j1},s_2,...,s_n)
\end{eqnarray*}
for $j=1,2,...$  .  Since $\widehat{\mu}(z,s_2,...,s_n)$ and $\psi(z,s_2,...,s_n)$ are analytic functions of $z$ in the whole complex
plane and they agree on the infinite bounded set $\{s_{j1}, j=1,2,...\}$ it follows that
\begin{eqnarray*}
 \widehat{\mu} (z,s_2,...,s_n) = \psi (z,s_2,...,s_n) \mbox{      } \forall  \mbox{      }  z \in \mathbb{C}.
\end{eqnarray*}
 Since this holds for all $(s_2,...,s_n)^T \in \mathbb{D}$ which is dense in $\mathbb{R}^{(n-1)}$ and both sides of the
equation are continuous on $\mathbb{R}^n$ we have
\begin{eqnarray*}
 \widehat{\mu} (s_1,s_2,...,s_n) = \psi (s_1,s_2,...,s_n)
\end{eqnarray*}
for all $(s_1,s_2,...,s_n)^T \in \mathbb{R}^n$.  This implies that $\mu$ is a gaussian measure with mean vector
$\bfm$ and covariance matrix $\Sigma$.
\end{proof}
\end{theo}

\noindent
{\bf Acknowledgement}

\vspace{0.3cm}
We thank S. Ramasubramanian for bringing our attention to the example of E. Nelson on page $99$ of W.Feller \cite{wf66} and
J. Stoyanov for his useful suggestions.

\end{document}